\documentclass[letter,10pt]{article} \usepackage[utf8]{inputenc} \usepackage{amsfonts}
\usepackage{amsthm} \usepackage{fancyhdr} \usepackage{comment} \usepackage{times}
\usepackage{amsmath} \usepackage{amsrefs} \usepackage{mathtools}
\usepackage{esint} \usepackage{changepage} \usepackage{amssymb} \usepackage{graphicx}
\usepackage{tikz} \usepackage{tikz-cd} \usepackage{titlesec} \usetikzlibrary{arrows}
\usepackage[mathscr]{euscript}  \theoremstyle{definition}
\newtheorem{theorem}{Theorem}[section]

\newtheorem{corollary}[theorem]{Corollary} 
\newtheorem{definition}[theorem]{Definition} 
 \newtheorem{lemma}[theorem]{Lemma}
 
\newtheorem{proposition}[theorem]{Proposition}

 \newcommand{\R}{\mathbb{R}} \newcommand{\C}{\mathbb{C}}
\newcommand{\Z}{\mathbb{Z}} \newcommand{\N}{\mathbb{N}}

\newcommand{\supp}{\operatorname{supp}} \newcommand{\dist}{\operatorname{dist}}
\newcommand{\diam}{\operatorname{diam}} 
 \newcommand{\Iso}{\operatorname{Iso}}
\title{Extremizers for Adjoint Restriction to a Pair of Reflected Paraboloids}
\author{James Tautges}

\begin{document}
	
\maketitle
	
\begin{abstract}
    We consider the adjoint restriction inequality associated to the hypersurface
    $\{(\tau, \xi) : \tau = \pm|\xi|^2, \;\xi \in \R^d\}$ at the Stein-Tomas exponent. Extremizers
    exist in all dimensions and extremizing sequences are precompact modulo symmetries conditional
    on a certain inequality, which we verify in the case $d \in \{1,2\}$.
\end{abstract}
	
\section{Introduction}

Fix $d \in \N$ and define
\begin{equation}
    \mathcal{E} f(t,x) = \int e^{i (t,x)(|\xi|^2, \xi)} f(\xi) d\xi.
\end{equation}
The adjoint Fourier restriction conjecture for the paraboloid is that
\begin{equation}\label{first sharp const statement}
    \sup_{f \in L^p} \frac{\|\mathcal{E}f\|_q}{\|f\|_p} = A_p < \infty
\end{equation}
for $q > p$ and $q = \frac{d+2}{d}p'$. We will call $p$ and $q$ for which \eqref{first sharp const
  statement} holds ``valid."

Beyond boundedness, one might ask if there exists $f \in L^p$ such that
$\|\mathcal{E}f\|_q/\|f\|_p = A_p$ and furthermore, if sequences $\{f_n\} \subset L^p$ such that
$\lim \|\mathcal{E}f_n\|_q/\|f_n\|_p = A_p$ have any special properties.

This paper answers these questions for the adjoint restriction inequality for the union of a pair of
reflected paraboloids. We prove that extremizers exist in all dimensions whenever the normal
conjecture holds, and that the previously mentioned sequences are precompact modulo certain
symmetries of the new operator provided $p=2$ and the operator norm is ``large enough.'' Finally, we
verify the hypothesis on the operator norm for $d \in \{1,2\}$.

For the standard single paraboloid in lower dimensions ($d \in \{1,2\}$) there are particularly
satisfactory answers. Radial Gaussians are the unique extremizers for $\mathcal{E}$ at the
Stein-Tomas exponent (\cite{Foschi 2007}, \cite{HZ}, \cite{BBCH}), and it is conjectured that this
is true for all dimensions. Christ and Quilodr\'an (\cite{Christ Quilodran}) proved that this is
only possible for $p \in \{1,2\}$.

In higher dimensions or away from the Stein-Tomas exponent, the method of profile decompositions and
mass concentration have been the most successful (\cite{Bourgain}, \cite{MV}, \cite{BV}, \cite{CK},
\cite{KV}, \cite{Shao 2009}, \cite{Stovall}). Since the method analyzes extremizing sequences, it
also yields precompactness of extremizing sequences modulo symmetries of the operator. The profile
decomposition approach has yielded the existence of extremizers in all dimensions but without an
explicit characterization (\cite{Shao 2009}). Stovall extended this approach in \cite{Stovall} to
prove that extremizers exist in all dimensions for all valid exponents by truncating $L^p$ functions
in a controlled way and then applying the $L^2$ theory. Our proof begins with a profile
decomposition mainly modelled after Killip and Vi\c san's formulation in \cite{KV}.

There are similar results for other restriction inequalities, for example restriction to
$S^d$. Extremizers exist and are characters multiplied by constants for $d \in \{1,2\}$ at the
Stein-Tomas exponent (\cite{Shao 2016}, \cite{Christ Shao 2}, \cite{Foschi 2015}) as well as for
some higher dimensions and even exponents (\cite{CO 2015}, \cite{OQ 2021}). The results establishing
the precompactness of extremizing sequences of Christ and Shao in \cite{Christ Shao 1} ($d=2$) and
Frank, Lieb, and Sabin in \cite{FLS} ($d > 2$ conditional on a conjecture) are particularly relevant
to our setting. Specifically, we borrow the strategy of relating the mass of individual bubbles to
the value of the operator norm as well as some computations from \cite{FLS}.

In this article, we address the same questions of existence of extremizers and characterization of
extremizing sequences for the adjoint restriction inequality associated to the union of a pair of
reflected paraboloids. Let $P = \{(\tau, \xi) : \tau = |\xi|^2\}$ and
$P_- = \{(\tau, \xi) : \tau = -|\xi|^2\}$. We pull back $d$-dimensional Lebesgue measure through the
projection map $\R \times \R^d \rightarrow \R^d$ to a measure $d\sigma$ supported on $P$ and
$d\sigma'$ supported on $P_-$. We will analyze the operator
\begin{multline*}
    \mathcal{E}_\pm (f,g)(t,x) := \mathcal{E}f(t,x) + \mathcal{E}_- g(t,x) \\
    := \int_P e^{i(t,x)(\tau, \xi)}f(\xi)d\sigma(\tau, \xi) + \int_{P_-} e^{i(t,x)(\tau, \xi)}
    g(\xi)d\sigma'(\tau, \xi).
\end{multline*}
It can also be expressed as
\begin{equation*}
    \mathcal{E}f(t,x) + \mathcal{E}_-g(t,x) = \int_{\R^d} e^{i(t,x)\cdot (|\xi|^2, \xi)}f(\xi) + e^{i(t,x)\cdot(-|\xi|^2, \xi)}g(\xi)d\xi.
\end{equation*}
When we write $f(\xi)$ or just $f$, we mean the function $\R^d \rightarrow \C$, whereas $fd\sigma$
will represent the measure $fd\xi$ lifted to $P$. Similarly, $gd\sigma'$ will represent the measure
$gd\xi$ lifted to $P_-$.

\subsection*{Symmetries and Definitions}
Let $\mathbf{S}_+ \subset \Iso(L^p(\R^d))$ and $\mathbf{T}_+ \subset \Iso(L^q(\R^{d+1}))$ be the
subgroups generated by
\begin{equation*}
    \begin{array}{lll}
      & Sf(\xi) & T\mathcal{E}f(t,x)\\
      \text{Scaling} & \lambda^{d/p}f(\lambda \xi) & \lambda^{-(d+2)/q}\mathcal{E}f(\lambda^{-2}t, \lambda^{-1} x) \\
      \text{Frequency Translation} & f(\xi - \xi') & e^{i(t|\xi'|^2 + x\cdot \xi')}\mathcal{E}f(t, x+2t\xi') \\
      \text{Spacetime Translation} & e^{i(t_0, x_0)(|\xi|^2, \xi)}f(\xi) & \mathcal{E}f(t + t_0, x + x_0).
    \end{array}
\end{equation*}
They are distinguished by the fact that
$\mathcal{E} \circ \mathbf{S}_+ = \mathbf{T}_+ \circ \mathcal{E}$ and each element generates
non-compact orbits in $L^p$ or $L^q$.

Using these generators, we can write any symmetry $S \in \mathbf{S}_+$ as
\begin{equation}\label{canonical symmetry S+}
    Sf(\xi) = \lambda^{d/p}e^{i(t_0, x_0)(|\lambda\xi - \xi'|^2, \lambda\xi - \xi')}f(\lambda\xi - \xi'),
\end{equation}
and the corresponding $T \in \mathbf{T}_+$ as
\begin{equation}\label{canonical symmetry T+}
    TF(t,x) = \lambda^{-(d+2)/q}e^{i(\lambda^{-2}t|\xi'|^2 + \lambda^{-1}x\cdot\xi')}F(\lambda^{-2}t + t_0, \lambda^{-1}x + x_0 + 2\lambda^{-2}t\xi'),
\end{equation}
for some $\lambda \in \R^+$, $(t_0, x_0) \in \R \times \R^d$, and $\xi' \in \R^d$.

Let $\mathbf{S}_- \subset \Iso(L^p(\R^d))$ and $\mathbf{T}_- \subset \Iso(L^q(\R^{d+1}))$ be the
subgroups generated by
\begin{equation*}
    \begin{array}{lll}
      & Sg(\xi) & T\mathcal{E}_-g(t,x)\\
      \text{Scaling} & \lambda^{d/p}g(\lambda \xi) & \lambda^{-(d+2)/q}\mathcal{E}_- g(\lambda^{-2}t, \lambda^{-1} x) \\
      \text{Frequency Translation} & g(\xi - \xi') & e^{i(-t|\xi'|^2 + x\cdot \xi')}\mathcal{E}_- g(t, x-2t\xi') \\
      \text{Spacetime Translation} & e^{i(t_0, x_0)(-|\xi|^2, \xi)}g(\xi) & \mathcal{E}g(t + t_0, x + x_0).
    \end{array}
\end{equation*}
Note that $\mathbf{T}_- \circ \mathcal{E}_- = \mathcal{E}_- \circ \mathbf{S}_-$ here as well.

Write any $S \in \mathbf{S}_-$ as
\begin{equation}\label{canonical symmetry S-}
    S g(\xi) = \lambda^{d/p}e^{i(-t_0,x_0)(|\lambda\xi - \xi'|^2, \lambda\xi - \xi')}g(\lambda\xi - \xi')
\end{equation}
for some $\lambda \in \R^+$, $(t_0, x_0) \in \R \times \R^d$, and $\xi' \in \R^d$. Denote its
matching $T \in \mathbf{T}_-$ by
\begin{multline}\label{canonical symmetry T-}
    T G(t,x) \\
    = \lambda^{-(d+2)/q} e^{i(-\lambda^{-2}t|\xi'|^2 + \lambda^{-1}x\cdot \xi')}
    G(\lambda^{-2}t+t_0, \lambda^{-1}x+x_0 - 2\lambda^{-2}t\xi').
\end{multline}

We also need notation to indicate whether or not two sequences of symmetries are asymptotically
orthogonal in a weak sense.

\begin{definition}\label{sym decouple condition}
    Let $\{S_n\}, \{S_n'\} \subset \mathbf{S}_+$ and $\{R_n\}, \{R_n'\} \subset \mathbf{S}_-$ be the
    symmetries associated with the parameters $(\lambda_n, t_n, x_n, \xi_n)$,
    $(\lambda_n', t_n', x_n', \xi_n')$, $(\kappa_n, s_n, y_n, \eta_n)$, and
    $(\kappa_n', s_n', y_n', \eta_n')$ respectively. Then we say $\{S_n\}\perp \{R_n\}$ ($\{S_n\}$
    is asymptotically orthogonal to $\{R_n\}$) if
    \begin{enumerate}
        \item $\lim \frac{\lambda_n}{\kappa_n} \in \{0, \infty\}$; or
        \item $\lim |\frac{\lambda_n}{\kappa_n}\eta_n + \xi_n| = \infty$; or
        \item
        $\lim |t_n - (\frac{\lambda_n}{\kappa_n})^2 s_n| + |x_n - \frac{\lambda_n}{\kappa_n}y_n -
        2\frac{\lambda_n}{\kappa_n}s_n(\eta_n + \frac{\lambda_n}{\kappa_n}\xi_n)| = \infty$.
    \end{enumerate}
    We say $\{S_n\}\perp \{S_n'\}$ if
    \begin{enumerate}
        \item $\lim \frac{\lambda_n}{\lambda_n'} \in \{0, \infty\}$; or
        \item $\lim |\frac{\lambda_n}{\lambda_n'}\xi_n' - \xi_n| = \infty$; or
        \item
        $\lim |t_n - (\frac{\lambda_n}{\lambda_n'})^2 t_n'| + |x_n -
        \frac{\lambda_n}{\lambda_n'}x_n' + 2\frac{\lambda_n}{\lambda_n'}t_n'(\xi_n' -
        \frac{\lambda_n}{\lambda_n'}\xi_n)| = \infty$.
    \end{enumerate}
    We define $\{R_n\}\perp \{R_n'\}$ in the same way.
\end{definition}

\subsection*{Results}

Let
\[
    A_p^\pm := \sup_{f,g\in L^p} \frac{\|\mathcal{E}f + \mathcal{E}_- g\|_q}{(\|f\|_p^p +
      \|g\|_p^p)^{1/p}}.
\]
Since
\begin{multline}\label{upper bound comp}
    \left\|\mathcal{E}f_1 + \mathcal{E}_- f_2\right\|_q \leq \left\|\mathcal{E}f_1\right\|_q + \left\| \mathcal{E}_- f_2 \right\|_q \\
    \leq A_p\left( \|f_1\|_p + \|f_2\|_p\right) \leq 2^{1/p'}A_p \left( \|f_1\|_p^p +
        \|f_2\|_p^p\right)^{1/p},
\end{multline}
$A_p^\pm < \infty$ provided $A_p < \infty$. The theorem assumes only that the exponents $p,q$ are
valid. The second is proven for the Stein-Tomas exponent.

\begin{theorem}\label{thm 1}
    \begin{enumerate}
        \item For all $d \in \N$,
        \[
            \left(\frac{1}{\sqrt{\pi}} \cdot
                \frac{\Gamma(\frac{q+1}{2})}{\Gamma(\frac{q+2}{2})}\right)^{1/q} 2^{1/p'} A_p \leq
            A_p^\pm < 2^{1/p'}A_p.
        \]
        \item If $d \in \{1,2\}$,
        \[
            \left(\frac{1}{\sqrt{\pi}} \cdot
                \frac{\Gamma(\frac{q+1}{2})}{\Gamma(\frac{q+2}{2})}\right)^{1/q} 2^{1/p'} A_p <
            A_p^\pm.
        \]
    \end{enumerate}
\end{theorem}

\begin{theorem}\label{thm 2}
    \begin{enumerate}
        \item Let $\{(f_n, g_n)\} \subset L^2 \times L^2$ be such that
        $\|f_n\|_2^2 + \|g_n\|_2^2 = 1$ for all $n$ and
        \[
            \lim_{n\rightarrow \infty} \|\mathcal{E}f_n + \mathcal{E}g_n\|_q = A_2^\pm.
        \]
        If
        \[
            \left(\frac{1}{\sqrt{\pi}} \cdot
                \frac{\Gamma(\frac{q+1}{2})}{\Gamma(\frac{q+2}{2})}\right)^{1/q} 2^{1/p'} A_2 <
            A_2^\pm,
        \]
        then there exist $\{Q_n\} \subset \mathbf{S}_+ \cap \mathbf{S}_-$ and $f,g \in L^2$ such
        that along a subsequence,
        \begin{enumerate}
            \item $\|f - Q_nf_n\|_2 \rightarrow 0$;
            \item $\|g - Q_ng_n\|_2 \rightarrow 0$; and therefore
            \item $\|\mathcal{E}f + \mathcal{E}_- g\|_q = A_2^\pm$.
        \end{enumerate}
        \item For all $d \in \N$, there exist $f, g \in L^2$ normalized so that
        $\|f\|_2^2 + \|g\|_2^2 = 1$ such that
        \[
            \|\mathcal{E}f + \mathcal{E}_- g\|_q = A_2^\pm.
        \]
    \end{enumerate}
\end{theorem}

The theorems immediately imply the following corollary.

\begin{corollary}
    For $d \in \{1,2\}$, extremizing sequences for $\mathcal{E}_\pm$ are precompact modulo the
    action of $\mathbf{S}_+ \cap \mathbf{S}_-$.
\end{corollary}

Our primary tool is the identity $\mathcal{E}_- f = \overline{\mathcal{E} \widetilde{f}}$ where
$\widetilde{f}(\xi) = \overline{f}(-\xi)$. For an extremizing sequence $\{(f_n, g_n)\}$, we apply a
standard profile decomposition modeled on the one from \cite{KV} to each sequence independently. We
then bound the interactions between bubbles whose symmetries are asymptotically orthogonal in the
sense of Definition \ref{sym decouple condition}. From this we deduce that almost all of the $L^2$
mass of $\{f_n\}$ and $\{g_n\}$ come from single bubbles whose symmetries are not asymptotically
orthogonal.

There are two cases. Either the bubbles are translated to infinity in frequency space, or they
remain bounded. An inequality from \cite{FLS} shows that under the hypothesis of Theorem \ref{thm 2}
part 1, we must be in the second case and therefore we can extract a convergent subsequence. We
verify this condition computationally for $d \in \{1,2\}$. If the bubbles translate off to infinity,
no sequence of symmetries can make them both converge. However, we know the operator norm exactly in
this case and are able to construct extremizers for $\mathcal{E}_\pm$ from extremizers for
$\mathcal{E}$, completing the proof of Theorem \ref{thm 2}.

In this paper, we can only calculate useful lower bounds on the operator norm of
$\mathcal{E}_\pm: \ell^p(L^p) \rightarrow L^q$ in dimensions where we know extremizers exist for
$\mathcal{E}: L^p \rightarrow L^q$ and have a particularly simple form, such as Gaussians. This is
conjectured to hold in all dimensions at the Stein-Tomas exponent, but based on the result of Christ
and Quilodr\'an (\cite{Christ Quilodran}), it seems unlikely that this will be possible for other
valid exponents. In addition, without further simplification the numerical computation must be run
for each dimension individually rather than once for the general case.

\subsection*{Acknowledgements}

This project was suggested and overseen by Betsy Stovall and supported in part by NSF
DMS-1653264. The author would like to thank her for many helpful conversations and invaluable
guidance in the writing of this paper.

\section{Bounds on $A_p^\pm$}

The purpose of this section is to prove Theorem \ref{thm 1} part 1. The first lemma resembles an
observation of Allaire (\cite{Allaire}).

\begin{lemma}\label{lambda lemma}
    Let $g \in L^q(\R^{1+d})$ and let $\varepsilon > 0$ be sufficiently small. Then there exists
    $\lambda_0 > 0$ such that
    \begin{equation*}
        \left|\int|\Im e^{i(-t|\eta|^2 + x\cdot\eta)}g|^qdtdx - \frac{1}{2\pi}\int_0^{2\pi} \int |\Im e^{i\theta}g|^qdtdxd\theta\right| < \varepsilon
    \end{equation*}
    for all $|\eta| > \lambda_0$.
\end{lemma}

\begin{proof}
    Let $g' \in C_c^\infty(\R^{d+1})$ be such that $\|g-g'\|_q^q < \varepsilon/4$. By the triangle
    inequality,
    \begin{multline*}
        \left| \int |\Im e^{i(t,x)\cdot(-|\eta|^2, \eta)}g|^q dtdx - \frac{1}{2\pi}\int_0^{2\pi}\int |\Im e^{i\theta}g|^qdtdxd\theta\right| \\
        \leq \left| \int |\Im e^{i(t,x)\cdot(-|\eta|^2, \eta)}g'|^q dtdx -
            \frac{1}{2\pi}\int_0^{2\pi}\int |\Im e^{i\theta}g'|^qdtdxd\theta\right| +
        \frac{\varepsilon}{2}.
    \end{multline*}
    Hence it suffices to prove the result for $g \in C_c^\infty$.
	
    Take $g \in C_c^\infty(\R^{d+1})$. Let $X = \lambda_0^{-1}\Z$ and $I = [0,1)$. Since $g$ is
    smooth and compactly supported, we can take $\lambda_0 > \varepsilon^{-1}$ large enough that
    \[
	\int |\Im e^{i(t,x)\cdot(-|\eta|^2,\eta)}g|^qdtdx \\
	= \int_{\R^d} \sum_{\alpha \in X} \int_{\alpha + \lambda_0^{-1}I} |\Im
        e^{i(t,x)\cdot(-|\eta|^2, \eta)}g(\alpha, x)|^q dtdx + O(\varepsilon),
    \]
    \[
	\int \frac{1}{2\pi}\int_0^{2\pi} |\Im e^{i\theta}g|^qd\theta dtdx = \int_{\R^d} \sum_{\alpha
          \in X} \frac{\lambda_0^{-1}}{2\pi} \int_0^{2\pi} |\Im e^{i\theta}g(\alpha, x)|^q d\theta
        dx + O(\varepsilon),
    \]
    and
    \[
	\int |g|^q dtdx = \int_{\R^d} \sum_{\alpha \in X} \lambda_0^{-1}|g(\alpha, x)|^q dx +
        O(\varepsilon).
    \]
    Furthermore, by the change of variables $\theta = t|\eta|^2 - x\cdot\eta$,
    \begin{align*}
      \int_{\alpha + \lambda_0^{-1}I} &|\Im e^{i(t,x)\cdot(-|\eta|^2, \eta)}g(\alpha, x)|^q dt \\
                                      &= |\eta|^{-2} \int_{|\eta|^2 \alpha - x\cdot\eta}^{|\eta|^2 \alpha - x\cdot\eta + \lambda_0^{-1}|\eta|^2} |\Im e^{-i\theta} g(\alpha, x)|^q d\theta\\
                                      &= |\eta|^{-2} \left[O(1)|g(\alpha, x)|^q + \frac{\lambda_0^{-1}|\eta|^2}{2\pi}\int_0^{2\pi} |\Im e^{i\theta}g(\alpha, x)|^qd\theta \right] \\
                                      &= O(1)|\eta|^{-2}|g(\alpha, x)|^q + \frac{\lambda_0^{-1}}{2\pi}\int_0^{2\pi} |\Im e^{i\theta}g(\alpha, x)|^qd\theta.
    \end{align*}
    Plugging this into the sum and simplifying,
    \begin{align*}
      \int &|\Im e^{i(t, x)\cdot(-|\eta|^2, \eta)}g|^qdtdx \\
           &= \int \sum_{\alpha \in X} \left[O(1)|\eta|^{-2}|g(\alpha, x)|^q + \frac{\lambda_0^{-1}}{2\pi}\int_0^{2\pi} |\Im e^{i\theta}g(\alpha, x)|^qd\theta\right] dx + O(\varepsilon) \\
           &= O(1)|\eta|^{-2}\lambda_0 \|g\|_q^q + \int \frac{1}{2\pi} \int_0^{2\pi} |\Im e^{i\theta}g|^q d\theta dt dx + O(\varepsilon).
    \end{align*}
    Hence for $|\eta| > \lambda_0 > \varepsilon^{-1}$, we have
    \[
	\left|\int |\Im e^{i(t, x)\cdot(-|\eta|^2, \eta)}g|^qdtdx - \int \frac{1}{2\pi}
            \int_0^{2\pi} |\Im e^{i\theta}g|^q d\theta dt dx\right| = O(\varepsilon).
    \]
\end{proof}

\begin{lemma}\label{real imaginary balance}
    Let $\varepsilon > 0$ be sufficiently small. Then there exists a $\delta$ such that for all
    $f \in L^p$ with $\|f\|_p = 1$ and $\|\mathcal{E}f\|_q > A_p - \varepsilon$,
    $\|\Im \mathcal{E}f\|_q > \delta$. Furthermore,
    $\lim_{\varepsilon \rightarrow 0} \delta =: C_p > 0$.
\end{lemma}

\begin{proof}
    Suppose we had a sequence $\{f_n\}$ with $\|f_n\|_p = 1$ such that
    $\|\Im \mathcal{E}f_n\|_q \rightarrow 0$ and $\|\mathcal{E}f_n\|_q \rightarrow A_p$. Extremizing
    sequences are precompact modulo symmetries (\cite{Stovall}), so there exists $f \in L^p$ and
    symmetries $\{(S_n, T_n)\}$ such that $\|\mathcal{E}f\|_q = A_p$ and $S_nf_n \rightarrow f$
    along a subsequence.
	
    We want to prove that there exists a function $g$ such that $\|\Im \mathcal{E}g\|_q = 0$. The
    quantity $\|\Im \mathcal{E}f\|_q$ is invariant under scaling and spacetime translation of $f$ so
    without loss of generality assume
    $\mathcal{E}f_n = e^{i(t,x)\cdot(|\eta_n|^2, \eta_n)}\mathcal{E}f(t,x+2t\eta_n)$. By hypothesis
    $\|\Im \mathcal{E}f_n\|_q \rightarrow 0$ so by a change of variables,
    $\|\Im e^{(t,x)\cdot(-|\eta_n|^2,\eta_n)}\mathcal{E}f\|_q \rightarrow 0$.
	
    Since $\mathcal{E}f \not\equiv 0$,
    $\int_0^{2\pi}\int |\Im e^{i\theta}\mathcal{E}f|^qdtdxd\theta \neq 0$. By Lemma \ref{lambda
      lemma} there exists $\varepsilon > 0$ and $\lambda_0 > 0$ such that
    $\|\Im e^{i(t,x)\cdot(-|\eta|^2, \eta)}\mathcal{E}f\|_q > \varepsilon$ for all
    $|\eta| > \lambda_0$. Therefore $\limsup |\eta_n| < \lambda_0$. The set
    $\{|\eta| \leq \lambda_0\}$ is compact, so by passing to a subsequence there exists an $\eta_0$
    such that $\eta_n \rightarrow \eta_0$. Since
    $\|\Im e^{i(t,x)\cdot(-|\eta|^2, \eta)}\mathcal{E}f\|_q$ is continuous in $\eta$ by the
    dominated convergence theorem,
    $\|\Im e^{i(t,x)\cdot(-|\eta_0|^2, \eta_0)} \mathcal{E}f\|_q = 0$. Changing variables and moving
    the frequency translation through $\mathcal{E}$, we have $g$ a frequency translation of $f$ such
    that $\|\Im \mathcal{E}g\|_q = 0$.
	
    However, this is impossible. Examining the Schr\"odinger equation, we find that it is equivalent
    to
    \begin{equation*}
        \begin{cases}
            \partial_t \Im u - \Delta \Re u = 0,\\
            \partial_t \Re u + \Delta \Im u = 0.
        \end{cases}
    \end{equation*}
    Hence $\partial_t \Re \mathcal{E}g \equiv 0$ and $\Delta \Re \mathcal{E}g \equiv 0$. By Weyl's
    lemma $\Re \mathcal{E}g(t, \cdot)$ agrees almost everywhere with a harmonic function, but since
    the time derivative is also zero, $\Re \mathcal{E}g(t, \cdot)$ is the same non-zero harmonic
    function for all $t$. Such functions are not in $L^q$.
	
    The above implies that there does not exist a sequence
    $\{(\|\mathcal{E}f_n\|_q, \|\Im\mathcal{E}f_n\|_q)\}$ in
    $A := \{(\|\mathcal{E}f\|_q, \|\Im \mathcal{E}f\|_q) : \|f\|_p = 1\} \subset [0, A_p] \times [0,
    A_p]$ that converges to $(A_p, 0)$. Then
    \begin{equation*}
        C_p := \liminf_{\varepsilon \rightarrow 0} \big( \inf \{\|\Im\mathcal{E}f\|_q : \|f\|_p = 1, \|\mathcal{E}f\|_q > A_p - \varepsilon\}\big) > 0.
    \end{equation*}
\end{proof}

\begin{proof}[Proof of Theorem \ref{thm 1} part 1]
    First, we prove that
    \[
        \left(\frac{1}{\sqrt{\pi}} \cdot
            \frac{\Gamma(\frac{q+1}{2})}{\Gamma(\frac{q+2}{2})}\right)^{1/q} 2^{1/p'} A_p \leq
        A_p^\pm.
    \]
    
    Let $f \in L^p$ be an extremizer for $\mathcal{E}$ normalized so that $\|f\|_p = 1$. Let
    $e_1 \in \R^d$ be the first unit vector and define sequences
    \[
	f_n(\xi) := f(\xi - ne_1) \quad \text{and} \quad g_n(\xi) := f(-\xi - ne_1).
    \]
    Then by Lemma \ref{lambda lemma} and \cite[Lemma 6.1]{FLS},
    \begin{align*}
      \lim &\int \left| \mathcal{E}f_n + \mathcal{E}_-g_n\right|^qdtdx \\
           &= \lim \int \left| \mathcal{E}f_n + \overline{\mathcal{E}\widetilde{g_n}}\right|^q \\
           &= \lim \int \left| e^{i(t,x)\cdot(n^2, ne_1)}\mathcal{E}f(t, x + 2nte_1) + e^{i(t,x)\cdot(-n^2, -ne_1)}\overline{\mathcal{E}f}(t, x + 2nte_1)\right|^q \\
           &= \lim \int \left| e^{i(t,x)\cdot(2n^2, 2ne_1)}\mathcal{E}f(t, x + 2nte_1) + \overline{\mathcal{E}f}(t, x + 2nte_1)\right|^q \\
           &= \lim \int \left| e^{i(t,x)\cdot(-2n^2, 2ne_1)}\mathcal{E}f(t, x) + \overline{\mathcal{E}f}(t, x)\right|^q \\
           &= \frac{1}{2\pi} \int_0^{2\pi} \int \left| e^{i\omega}|\mathcal{E}f| + |\overline{\mathcal{E}f}|\right|^q \\
           &= \frac{2^q}{\sqrt{\pi}} \cdot \frac{\Gamma(\frac{q+1}{2})}{\Gamma(\frac{q+2}{2})}A_p^q.
    \end{align*}
    Since $(\|f_n\|_p^p + \|g_n\|_p^p)^{1/p} = 2^{1/p}$, this proves the first inequality.
	
    Next, following the proof of \cite[Lemma 6.1]{FLS} let
    \[
	\phi(t) := \frac{1}{\pi}\int_0^\pi (1+t\cos \theta)^{q/2}d\theta.
    \]
    Then
    \begin{align*}
      \phi'(t) &= \frac{q}{2\pi} \int_0^\pi (1+t\cos \theta)^{(q-2)/2}\cos\theta d\theta \\
               &= \frac{q}{2\pi} \int_0^{\pi/2} \left((1+t\cos\theta)^{(q-2)/2} - (1-t\cos\theta)^{(q-2)/2}\right)\cos\theta d\theta \\
               &> 0.
    \end{align*}
    Since $\phi(0) = 1$ and
    \[
	\phi(1) = \frac{2^{q/2}}{\sqrt{\pi}} \frac{\Gamma(\frac{q+1}{2})}{\Gamma(\frac{q+2}{2})}
    \]
    using the integral representation of the beta function, the second inequality follows.

    Second, we prove that
    \[
        A_p^\pm < 2^{1/p'}A_p.
    \]
    
    Note that for a nontrivial sequence $\{(f_n,g_n)\}$ to achieve
    \begin{equation*}
        \lim_{n \rightarrow \infty} \frac{\|\mathcal{E}f_n + \mathcal{E}_- g_n\|_q}{(\|f_n\|_p^p + \|g_n\|_p^p)^{1/p}} = 2^{1/p'}A_p,
    \end{equation*}
    it must approach equality for every inequality in \eqref{upper bound comp}. Therefore, aiming
    for a contradiction, we assume that there exists a sequence $\{(f_n, g_n)\}$ normalized in
    $\ell^p(L^p)$ such that $\|\mathcal{E}f_n\|_q \rightarrow A_p$,
    $\|\mathcal{E}g_n\|_q \rightarrow A_p$, and
    $\|\mathcal{E}f_n - \mathcal{E}_- g_n\|_q \rightarrow 0$.
    Expanding the integral using the identity
    $\mathcal{E}_- g = \overline{\mathcal{E}\widetilde{g}}$,
    \begin{equation*}
        \|\mathcal{E}f_n - \mathcal{E}_- g_n\|_q^q = \int \left( |\Re\mathcal{E}f_n - \Re\mathcal{E}\widetilde{g}_n|^2 + |\Im\mathcal{E}f_n + \Im\mathcal{E}\widetilde{g}_n|^2\right)^{q/2}.
    \end{equation*}
    Hence $\|\Im\mathcal{E}f_n + \Im\mathcal{E}\widetilde{g}_n\|_q \rightarrow 0$. If we multiply
    $f_n$ and $g_n$ by $i$ and use the linearity of $\mathcal{E}$ and $\mathcal{E}_-$, we see that
    \begin{align*}
      \|\mathcal{E}f_n - \mathcal{E}_- g_n\|_q^q &= \|\mathcal{E}(if_n) - \mathcal{E}_- (ig_n)\|_q \\
                                                 &= \int \left( |-\Im\mathcal{E}f_n + \Im\mathcal{E}\widetilde{g}_n|^2 + |\Re\mathcal{E}f_n + \Re\mathcal{E}\widetilde{g}_n|^2\right)^{q/2}
    \end{align*}
    and hence $\|\Im \mathcal{E}f_n - \Im\mathcal{E}\widetilde{g}_n\|_q \rightarrow 0$ as
    well. However, by the triangle inequality and Lemma \ref{real imaginary balance}
    \begin{equation*}
        \begin{split}
            \lim_{n\rightarrow \infty} \|\Im \mathcal{E}f_n - \Im\mathcal{E}\widetilde{g}_n\|_q &\geq \lim 2\|\Im\mathcal{E}\widetilde{g}_n\|_q - \|\Im\mathcal{E}f_n + \Im\mathcal{E}\widetilde{g}_n\|_q \\
            &= \lim 2\|\Im\mathcal{E}\widetilde{g}_n\|_q \\
            &\geq 2C_pA_p,
        \end{split}
    \end{equation*}
    which is a contradiction. Therefore no such sequence $\{(f_n,g_n)\}$ exists and
    $A_p^\pm < 2^{1/p'}A_p$.
\end{proof}

\section{Bubble pairing and orthogonality}

Let $p = 2$ and $q = \frac{2(d+2)}{d}$.

We will use the following formulation of the $L^2$ profile decomposition for the Schr\"odinger
equation.

\begin{proposition}{\cite[Theorem 4.7]{KV}}\label{profile decomp}
    Let $\{f_n\} \subset L^2(\R^d)$ be bounded. Then after passing to a subsequence, there exist
    \begin{enumerate}
        \item $J^* \in \N \cup \{\infty\}$;
        \item functions $\phi^j \in L^2$ for all $j < J^*$;
        \item remainders $w_n^J \in L^2$ for all $J < J^*$;
        \item and sequences $\{\lambda_n^j\} \subset \R^+$, $\{(t_n^j, x_n^j)\} \subset \R^{d+1}$,
        and $\{\xi_n^j\} \subset \R^d$ defining $\{S_n^j\} \subset \mathbf{S}_+$ for all $j < J^*$
        and all $n$
    \end{enumerate}
    such that for all $J < J^*$
    \begin{equation}\label{profile def}
        f_n = \sum_{j=1}^J (\lambda_n^j)^{d/p} e^{i(t_n^j, x_n^j)\cdot(|\lambda_n^j \xi - \xi_n^j|^2, \lambda_n^j \xi - \xi_n^j)} \phi^j(\lambda_n^j\xi - \xi_n^j) + w_n^J,
    \end{equation}
    \begin{equation}\label{remainder ext}
        \lim_{J \rightarrow J^*} \limsup_{n\rightarrow \infty} \|\mathcal{E}w_n^J\|_q = 0,
    \end{equation}
    \begin{equation}\label{profile ortho}
        \sup_J \lim_{n\rightarrow\infty} \left[ \|f_n\|_2^2 - \sum_{j=1}^J \|\phi^j\|_2^2 - \|w_n^J\|_2^2\right] = 0,
    \end{equation}
    \begin{equation}\label{ext ortho}
        \sup_J \lim_{n \rightarrow \infty} \left[ \|\mathcal{E}f_n\|_{q}^{q} - \sum_{j=1}^J \|\mathcal{E}\phi^j\|_{q}^{q} - \|\mathcal{E}r_n^J\|_{q}^{q}\right] = 0,
    \end{equation}
    and for all $j \neq k$,
    \begin{equation}
        \{S_n^j\} \perp \{S_n^k\}.
    \end{equation}
\end{proposition}

The same holds for $\mathcal{E}_-$ with \eqref{profile def} replaced by
\begin{equation}\label{refl profile def}
    g_n = \sum_{j=1}^J (\kappa_n^j)^{d/p} e^{i(-s_n^j, y_n^j)\cdot(|\kappa_n^j \xi - \eta_n^j|^2, \kappa_n^j \xi - \eta_n^j)} \psi^j(\kappa_n^j\xi - \eta_n^j) + r_n^J.
\end{equation}

Let $f_n, g_n$ be bounded sequences in $L^2$ and apply the profile decompositions of \eqref{profile
  def} and \eqref{refl profile def} respectively. If the $J^*$ are different, take the larger of the
two and pad the other sum with zero bubbles.

\begin{lemma}\label{one or the other}
    Let $\{S_n^j\}$ and $\{R_n^j\}$ be the sequences of symmetries associated with the
    decompositions of $\{f_n\}$ and $\{g_n\}$ respectively. If $\{S_n^j\}\not\perp \{R_n^k\}$, then
    $\{S_n^j\} \perp \{R_n^{k'}\}$ for all $k' \neq k$.
\end{lemma}

The proof is a long, unenlightening calculation, so we omit it for readability.

\begin{lemma}\label{same parab decoupling}
    Let $\{S_n^j\}, \{S_n^{j'}\} \subset \mathbf{S}_+$. Let $\phi^j, \phi^{j'} \in L^2$. If
    $\{S_n^j\} \perp \{S_n^{j'}\}$, then
    \[
	\lim_n \int |\mathcal{E}S_n^j\phi^j||\mathcal{E}S_n^{j'}\phi^{j'}|^{q-1} +
        |\mathcal{E}S_n^j\phi^j|^{q-1}|\mathcal{E}S_n^{j'}\phi^{j'}| = 0.
    \]
\end{lemma}
\begin{proof}
    By the symmetry of the statement, we need only consider the first term in the limit. Let
    $\varepsilon > 0$ and let $\Phi^j \in C^\infty_{cpct}(\R^{d+1})$ be such that
    $\|\Phi^j - \mathcal{E}\phi^j\|_q = O(\varepsilon)$ (resp. $\Phi^{j'}$). Expanding, changing
    variables, and dealing with the remainders by H\"older,
    \begin{align*}
      &\int |\mathcal{E}S_n^j \phi^j| |\mathcal{E}S_n^{j'}\phi^{j'}|^{q-1} \\
      &= \int (\lambda_n^j)^{-(d+2)/q} (\lambda_n^{j'})^{-(d+2)/q'} |\Phi^j((\lambda_n^j)^{-2} t + t_n^j, (\lambda_n^j)^{-1}x + x_n^j + 2(\lambda_n^j)^{-2}t\xi_n^j)| \\
      &\qquad |\Phi^{j'}((\lambda_n^{j'})^{-2}t + t_n^{j'}, (\lambda_n^{j'})^{-1}x + x_n^{j'} + 2(\lambda_n^{j'})^{-2}t\xi_n^{j'}|^{q-1} + O(\varepsilon)\\
      &= \int (\lambda_n^j)^{-(d+2)/q}(\lambda_n^{j'})^{(d+2)/q}|\Phi^j(\left(\frac{\lambda_n^{j'}}{\lambda_n^j}\right)^2 t + t_n^j, \left(\frac{\lambda_n^{j'}}{\lambda_n^j}\right)x + x_n^j + 2\left(\frac{\lambda_n^{j'}}{\lambda_n^j}\right)^2 \xi_n^j)| \\
      &\qquad |\Phi^{j'}(t + t_n^{j'}, x + x_n^{j'} + 2t\xi_n^{j'})|^{q-1} + O(\varepsilon)\\
      &= \int \left(\frac{\lambda_n^{j'}}{\lambda_n^j}\right)^{(d+2)/q} |\Phi^j(\left(\frac{\lambda_n^{j'}}{\lambda_n^j}\right)^2 t + t_n^j, \left(\frac{\lambda_n^{j'}}{\lambda_n^j}\right)x + x_n^j + 2\left(\frac{\lambda_n^{j'}}{\lambda_n^j}\right)^2 \xi_n^j)| \\
      &\qquad |\Phi^{j'}(t + t_n^{j'}, x + x_n^{j'} + 2t\xi_n^{j'})|^{q-1} + O(\varepsilon).
    \end{align*}
    By H\"older's inequality, we see that
    \[
	\int |\mathcal{E}S_n^j\phi^j| |\mathcal{E}S_n^{j'}\phi^{j'}|^{q-1} \leq
        \left(\frac{\lambda_n^{j'}}{\lambda_n^j}\right)^{(d+2)/q} \|\Phi^j\|_\infty
        \|(\Phi^{j'})^{q-1}\|_{1} + O(\varepsilon)
    \]
    and
    \[
	\int |\mathcal{E}S_n^j\phi^j| |\mathcal{E}S_n^{j'}\phi^{j'}|^{q-1} \leq
        \left(\frac{\lambda_n^{j'}}{\lambda_n^j}\right)^{-(d+2)/q'} \|\Phi^j\|_1
        \|(\Phi^{j'})^{q-1}\|_\infty + O(\varepsilon).
    \]
    Therefore if $\lim \frac{\lambda_n^{j'}}{\lambda_n^j} \in \{0,\infty\}$ (i.e.\
    $\{S_n^j\}\perp \{S_n^{j'}\}$ is satisfied by condition 1),
    \[
	\int |\mathcal{E}S_n^j\phi^j||\mathcal{E}S_n^{j'}\phi^{j'}|^{q-1} = O(\varepsilon).
    \]
    Taking $\varepsilon \rightarrow 0$, this proves the claim.
	
    Now we may assume that, after passing to a subsequence in $n$,
    $\lim \frac{\lambda_n^{j'}}{\lambda_n^j} = R \in (0, \infty)$. By scaling, we may also assume
    that $\lambda_n^j = r_n^{-1}$ such that $\lim r_n = R$ and $\lambda_n^{j'} = 1$ for all $n$.
	
    Let $\varepsilon > 0$ and let $f,g \in C^\infty_{cpct}$ be such that
    $\|f - \phi^j\|_2 = O(\varepsilon)$ and $\|g - \phi^{j'}\|_2 = O(\varepsilon)$. By the
    boundedness of $\mathcal{E}$, the triangle inequality, and H\"older,
    \begin{align*}
      \int|\mathcal{E}S_n^j\phi^j||\mathcal{E}S_n^{j'}\phi^{j'}|^{q-1} &\leq \int |\mathcal{E}S_n^j f||\mathcal{E}S_n^{j'}g|^{q-1} + O(\varepsilon) \\
                                                                       &\leq \|(\mathcal{E}S_n^jf)(\mathcal{E}S_n^{j'}g)\|_{s_\delta} \|(\mathcal{E}S_n^{j'} g)^{q-2}\|_{s_\delta'} + O(\varepsilon) \\
                                                                       &= \|(\mathcal{E}S_n^j f)(\mathcal{E}S_n^{j'} g)\|_{s_\delta} \|\mathcal{E}S_n^{j'}g\|_{s_\delta'(q-2)}^{q-2} + O(\varepsilon),
    \end{align*}
    where $\delta > 0$ is sufficiently small (to be determined), $s_\delta = \frac{q + \delta}{2}$,
    and $t_\delta' = \frac{ds_\delta(q-2)}{d+2}$. As a function of $\delta$,
    $s_\delta'(q-2) = \frac{(q+\delta)(q-2)}{q+\delta-2}$ is continuous near zero and for
    $\delta=0$, $s_0'(q-2) = q$. Since the exponent pair $(2, \frac{2(d+2)}{d})$ is in the interior
    of the known range of restriction estimates (e.g.\ \cite{Tao}), there exists a $\delta_0 > 0$
    such that the operators $\mathcal{E}$ and $\mathcal{E}_-$ are bounded from $L^{t_{\delta_0}}$ to
    $L^{s_{\delta_0}'(q-2)}$. Let
    \[
        N_n := \frac{\dist(\supp S_n^j f, \supp S_n^{j'}g)}{\diam(\supp S_n^j f) + \diam(\supp
          S_n^{j'}g)}.
    \]
    This number is well-defined since $f$ and $g$ have compact support. Note that since $\{r_n\}$ is
    bounded, $N_n \sim \dist(\supp S_n^j f, \supp S_n^{j'} g)$ with some $n$-independent
    constants. Continuing the estimate using scaled bilinear restriction (\cite[Corollary 1.3]{Tao})
    and the scaling assumption,
    \begin{align*}
      \|(\mathcal{E}S_n^j f)(\mathcal{E}S_n^{j'}g)\|_{s_{\delta_0}} \|\mathcal{E}S_n^{j'}g\|_{s_{\delta_0}'(q-2)}^{q-2} &\lesssim N_n^{d-\frac{d+2}{s_{\delta_0}}} \|f\|_2 \|g\|_2 \|S_n^{j'}g\|_{t_{\delta_0}}^{q-2} \\
                                                                                                                        &\lesssim_{R, f, g} N_n^{d-\frac{d+2}{s_{\delta_0}}}.
    \end{align*}
    Since $s_{\delta_0} > \frac{d+2}{d}$ and $N_n \sim |r_n\xi_n^{j'} - \xi_n^j|$ for large enough
    $n$, taking $n\rightarrow \infty$ and $\varepsilon \rightarrow 0$ proves the claim when
    $\{S_n^j\}\perp \{S_n^{j'}\}$ is satisfied by 2 but not 1.
	
    We turn to the final case where neither 1 nor 2 is satisfied, but condition 3 is. We make the
    same assumptions on scaling as in the previous case. Changing variables,
    \begin{equation*}
        \int |\mathcal{E}S_n^j \phi^j||\mathcal{E}S_n^{j'}\phi^{j'}|^{q-1} = \int R^{\frac{d+2}{q}} |\mathcal{E}\phi^j (A_n(t,x) + b_n)||\mathcal{E}\phi^{j'}|^{q-1}
    \end{equation*}
    where
    \[
	A_n = \begin{pmatrix}
            r_n^2 & 0\\
            2r_n(r_n\xi_n^j - \xi_n^{j'}) & r_nI_d
	\end{pmatrix}
	\; \text{and} \; b_n = \begin{pmatrix}
            t_n^j - r_n^2t_n^{j'}\\
            x_n^j - r_nx_n^{j'} + 2r_nt_n^{j'}(\xi_n^{j'} - r_n\xi_n^j)
	\end{pmatrix}.
    \]
    Since condition 2 is not satisfied, we may pass to a subsequence in $n$ so that
    $\|A_n\|_{op} \sim 1$. Condition 3 implies $|b_n| \rightarrow \infty$, so
    $|A_n(t,x) + b_n| \rightarrow \infty$ uniformly on compact sets. Approximating
    $\mathcal{E}\phi^j$ and $\mathcal{E}\phi^{j'}$ in $L^q$ by functions $F,G \in C^\infty_{cpct}$,
    we see that
    \[
        \lim \int R^\frac{d+2}{q}|F(A_n(t,x) + b_n)||G|^{q-1} = 0
    \]
    by dominated convergence. The lemma follows by taking better and better approximations.
\end{proof}

\begin{lemma}\label{diff parab decoupling}
    Let $\{S_n^j\} \subset \mathbf{S}_+$ and $\{R_n^j\} \subset \mathbf{S}_-$. Let
    $\phi^j,\psi^j \in L^2$. If $\{S_n^j\}\perp \{R_n^j\}$, then
    \[
	\lim \int |\mathcal{E}S_n^j\phi^j||\mathcal{E}_-R_n^j\psi^j|^{q-1} +
        |\mathcal{E}S_n^j\phi^j|^{q-1}|\mathcal{E}_-R_n^j\psi^j| = 0.
    \]
\end{lemma}
\begin{proof}
    As usual, we can express $S_n^j$ and $R_n^j$ in the form of \eqref{canonical symmetry S+} and
    \eqref{canonical symmetry S-}:
    \begin{align*}
      S_n^j \phi^j(\xi) &= (\lambda_n^j)^{d/p} e^{i(t_n^j,x_n^j)\cdot(|\lambda_n^j\xi - \xi_n^j|^2, \lambda_n^j\xi - \xi_n^j)}\phi^j(\lambda_n^j\xi - \xi_n^j), \\
      R_n^j \psi^j(\xi) &= (\kappa_n^j)^{d/p} e^{i(-s_n^j, y_n^j)\cdot(|\kappa_n^j\xi - \eta_n^j|^2, \kappa_n^j\xi - \eta_n^j)}\psi^j(\kappa_n^j\xi - \eta_n^j).
    \end{align*}
    Define the sequence $\{U_n^j\} \subset \mathbf{S}_+$ by
    \[
	Q_n^j f(\xi) := (\kappa_n^j)^{d/p} e^{i(s_n^j, y_n^j)\cdot(|\kappa_n^j\xi + \eta_n^j|^2,
          \kappa_n^j\xi + \eta_n^j)}f(\kappa_n^j\xi + \eta_n^j).
    \]
    By a change of variables,
    \begin{align*}
      \overline{\mathcal{E}_- R_n^j f}(t,x) &= \int e^{-i(t,x)\cdot(-|\xi|^2, \xi)}\overline{R_n^jf}(\xi)d\xi \\
                                            &= \int e^{i(t,x)\cdot(|\xi|^2, -\xi)}(\kappa_n^j)^{d/p}e^{i(s_n^j, y_n^j)\cdot(|\kappa_n^j\xi - \eta_n^j|^2, -\kappa_n^j\xi + \eta_n^j)} \overline{f}(\kappa_n^j\xi - \eta_n^j)d\xi \\
                                            &= \int e^{i(t,x)\cdot(|\xi|^2, \xi)}(\kappa_n^j)^{d/p}e^{i(s_n^j, y_n^j)\cdot(|\kappa_n^j\xi + \eta_n^j|^2, \kappa_n^j\xi + \eta_n^j)}\overline{\widetilde{f}}(\kappa_n^j\xi + \eta_n^j)d\xi \\
                                            &= \mathcal{E}Q_n^j \overline{\widetilde{f}}(t,x).
    \end{align*}
    The symmetries $Q_n^j$ have the same dilation and spacetime translation as $R_n^j$, only with
    negative frequency translation. From this it is clear that
    $\{S_n^j\} \perp \{R_n^j\} \iff \{S_n^j\} \perp \{Q_n^j\}$ and
    $|\mathcal{E}_- R_n^j f| = |\mathcal{E} Q_n^j \overline{\widetilde{f}}|$. Thus the claim follows
    from the previous lemma.
\end{proof}

We will need the following elementary calculus lemma.

\begin{lemma}\label{calc inequality}
    There exists $C_n > 0$ such that for all $a_1, \dots, a_n \in \C$,
    \[
	\left| |\sum_{j=1}^n a_j|^q - \sum_{j=1}^n |a_j|^q\right| \leq C_n \sup_{1\leq j \neq j'
          \leq n} |a_j||a_{j'}|^{q-1}.
    \]
\end{lemma}
\begin{proof}
    First, let $a,b \in \C$ be such that $|a| \geq |b|$. Let $F(z) = z^q$. Then by the triangle
    inequality and the mean value theorem,
    \begin{align*}
      \left| |a+b|^q - |a|^q - |b|^q\right| &\leq \left| |a+b|^q - |a|^q\right| + |b|^q \\
                                            &\leq |b| \sup_{|z| \leq 2|a|} |F'(z)| + |b|^q \\
                                            &= q2^{q-1}|b||a|^{q-1} + |b|^q \\
                                            &\lesssim \max\{|a||b|^{q-1}, |a|^{q-1}|b|\}.
    \end{align*}
	
    Now take any $a_1, \dots, a_n \in \C$. We can expand telescopically and apply the base case to
    show that
    \begin{align*}
      \left| |\sum_{j=1}^n a_j|^q - \sum_{j=1}^n |a_j|^q\right| &= \left| \sum_{k=2}^n \left(|\sum_{j=1}^k a_j|^q - |\sum_{j=1}^{k-1} a_j|^q - |a_k|^q\right) \right| \\
                                                                &\leq \sum_{k=2}^n \left||\sum_{j=1}^k a_j|^q - |\sum_{j=1}^{k-1} a_j|^q - |a_k|^q \right| \\
                                                                &\lesssim \sup_{2 \leq k \leq n} \max\left\{|\sum_{j=1}^{k-1}a_j|^{q-1}|a_k|, |\sum_{j=1}^{k-1}a_j||a_k|^{q-1}\right\}.
    \end{align*}
    By the triangle inequality and the pseudo-triangle inequality
    $|a+b|^{q-1} \lesssim |a|^{q-1} + |b|^{q-1}$, we obtain the result.
\end{proof}

\begin{proposition}\label{special double decoupling}
    Let $\{(f_n,g_n)\} \subset L^2 \times L^2$ be bounded. Then there exist
    $J^* \in \N \cup \{\infty\}$ and decompositions
    \begin{equation*}
        f_n = \sum_{j=1}^J S_n^j\phi^j + r_n^J \quad \text{and} \quad g_n = \sum_{j=1}^J R_n^j\psi^j + w_n^J
    \end{equation*}
    for $J \leq J^*$ that satisfy all the conclusions in Proposition \ref{profile decomp} and such
    that there exist partitions $A_J \cup B_J = \{1,\dots,J\}$ such that
    \begin{equation}\label{refl ext ortho}
        \begin{split}
            \lim_{J\rightarrow J^*} \limsup_{n\rightarrow \infty} &\|\mathcal{E}f_n + \mathcal{E}_-g_n\|_q^q - \sum_{j\in A_J} \|\mathcal{E}S_n^j\phi^j + \mathcal{E}_-R_n^j\psi^j\|_q^q \\
            &- \sum_{j\in B_J} \|\mathcal{E}\phi^j\|_q^q - \sum_{j\in
              B_J}\|\mathcal{E}_-\psi^j\|_q^q - \|\mathcal{E}r_n^J\|_q^q - \|\mathcal{E}_-
            w_n^J\|_q^q = 0.
        \end{split}
    \end{equation}
    Furthermore, for $J_1 \leq J_2 \leq J^*$ we have $A_{J_1} \subset A_{J_2}$.
\end{proposition}
\begin{proof}
    Let $J < J^*$ and $j \leq J$. By Lemma \ref{one or the other}, there exists at most one
    $j' \leq J$ such that $\{S_n^j\}\not\perp \{R_n^{j'}\}$. If such a $j'$ exists we may assume
    that $j = j'$ by rearranging the sequence $\{\psi^j\}_1^J$ and adding $j$ to the set $A_J$. If
    no such $j'$ exists, we add $j$ to $B_J$. By the triangle inequality,
    \begin{multline}\label{dumb ineq 1}
        \Big|\Big.\|\mathcal{E}f_n + \mathcal{E}_-g_n\|_q^q - \sum_{j\in A_J} \|\mathcal{E}S_n^j\phi^j + \mathcal{E}_-R_n^j\psi^j\|_q^q \\
        - \sum_{j\in B_J} \|\mathcal{E}\phi^j\|_q^q - \sum_{j\in B_J}\|\mathcal{E}_-\psi^j\|_q^q - \|\mathcal{E}r_n^J\|_q^q - \|\mathcal{E}_- w_n^J\|_q^q\Big|\Big. \\
        \leq \int \Big|\Big. |\mathcal{E}w_n^J + \mathcal{E}_-r_n^J + \sum_{j=1}^J \mathcal{E}S_n^j\phi^j + \mathcal{E}_-R_n^j\psi^j|^q - \sum_{j \in A_J} |\mathcal{E}S_n^j\phi^j + \mathcal{E}_- R_n^j\psi^j|^q \\
        - \sum_{j \not\in A_J} \left(|\mathcal{E}S_n^j \phi^j|^q + |\mathcal{E}_- R_n^j
            \psi^j|^q\right) - |\mathcal{E}w_n^J|^q - |\mathcal{E}_- r_n^J|^q \Big|\Big. .
    \end{multline}
	
    Thanks to Lemma \ref{calc inequality}, Lemmas \ref{same parab decoupling} and \ref{diff parab
      decoupling}, the quasi-triangle inequality, H\"older, and the fact that $\mathcal{E}$ is
    bounded, we continue the estimate of \eqref{dumb ineq 1} with
    \begin{align*}
      \lim_{n\rightarrow \infty} \text{RHS} &\lesssim \lim_{n\rightarrow\infty} \sup_{j,j' \leq J, \;\epsilon_1 \neq \epsilon_2 \in \{1,q-1\}} \int |\mathcal{E}S_n^j\phi^j + \mathcal{E}_-R_n^{j'}\psi^{j'}|^{\epsilon_1}|\mathcal{E}w_n^J + \mathcal{E}_-r_n^J|^{\epsilon_2} \\
                                            &\lesssim \lim_{n\rightarrow \infty} \sup_{j, j' \leq J, \; \epsilon_1 \neq \epsilon_2 \in \{1,q-1\}} \|\mathcal{E}S_n^j\phi^j + \mathcal{E}_-R_n^j\psi^j\|_q^{\epsilon_1}\|\mathcal{E}w_n^J + \mathcal{E}_- r_n^J\|_q^{\epsilon_2} \\
                                            &\lesssim \lim_{n\rightarrow\infty} \|\mathcal{E}w_n^J\|_q + \|\mathcal{E} w_n^J\|_q^{q-1} + \|\mathcal{E}_- r_n^J\|_q + \|\mathcal{E}_- r_n^J\|_q^{q-1}.
    \end{align*}
    By \eqref{remainder ext}, the claim is proved.
\end{proof}

\section{Existence of Extremizers}

Now we use the profile decomposition and information about the operator norm to find a pair of
bubbles that accounts for the full $L^2$ mass in the limit.

\begin{proof}[Proof of Theorem \ref{thm 2} part 1]
    Assume that
    \[
        \left(\frac{1}{\sqrt{\pi}} \cdot
            \frac{\Gamma(\frac{q+1}{2})}{\Gamma(\frac{q+2}{2})}\right)^{1/q} 2^{1/p'} A_p < A_p^\pm
    \]
    as in the statement of the theorem.
        
    Let $\{(f_n,g_n)\} \in \ell^2(L^2)$ be such that
    \begin{equation*}
        \lim_{n\rightarrow \infty} \|\mathcal{E}f_n + \mathcal{E}_- g_n\|_q = A_2^\pm \quad \text{and} \quad \|f_n\|_2^2 + \|g_n\|_2^2 = 1
    \end{equation*}
    for all $n$. Using Proposition \ref{special double decoupling}, write
    \begin{equation*}
        f_n = \sum_{j=1}^J S_n^j \phi^j + r_n^J \quad \text{and} \quad g_n = \sum_{j=1}^J R_n^j \psi^j + w_n^J
    \end{equation*}
    for $J < J^*$. Let $\varepsilon > 0$ and let $J$ and $N$ be sufficiently large dependent on
    $\varepsilon$. Then by applying \eqref{ext ortho}, H\"older, and \eqref{profile ortho},
    \begin{equation*}
        \begin{split}
            (A_2^\pm)^q - \varepsilon &\leq \sum_{j\in A_J} \|\mathcal{E}S_n^j\phi^j + \mathcal{E}_- R_n^j\psi^j\|_q^q + \sum_{j\in B_J} (\|\mathcal{E}\phi^j\|_q^q + \|\mathcal{E}_- \psi^j\|_q^q) \\
            &\quad+ \|\mathcal{E}w_n^J\|_q^q + \|\mathcal{E}_- r_n^J\|_q^q \\
            &\leq \sup_{k \in A_J} \|\mathcal{E}S_n^k\phi^k + \mathcal{E}_- R_n^k\psi^k\|_q^{4/d} \sum_{j\in A_J} \|\mathcal{E}S_n^j\phi^j + \mathcal{E}_- R_n^j\psi^j\|_q^2\\
            &\quad + \sup_{k \in B_J} \max \{\|\mathcal{E}\phi^k\|_q^{4/d}, \|\mathcal{E}_- \psi^k\|_q^{4/d}\}\sum_{j\in B_J} (\|\mathcal{E}\phi^j\|_q^2 + \|\mathcal{E}_- \psi^j\|_q^2) + O(\varepsilon) \\
            &\leq (A_2^\pm)^2 \sup_{k \in A_J} \|\mathcal{E}S_n^k\phi^k + \mathcal{E}_- R_n^k\psi^k\|_q^{4/d} \sum_{k \in A_J} (\|\phi^j\|_2^2 + \|\psi^j\|_2^2) \\
            &\quad + A_2^2 \sup_{k \in B_J} \max \{\|\mathcal{E}\phi^k\|_q^{4/d}, \|\mathcal{E}_- \psi^k\|_q^{4/d}\}\sum_{j\in B_J} (\|\phi^j\|_2^2 + \|\psi^j\|_2^2) + O(\varepsilon)\\
            &\leq \max \left\{(A_2^\pm)^2 \|\mathcal{E}S_n^{j_0}\phi^{j_0} + \mathcal{E}_- R_n^{j_0}\psi^{j_0}\|_q^{4/d}, A_2^2 \|\mathcal{E}\phi^{j_1}\|_q^{4/d}, A_2^2 \|\mathcal{E}_-\psi^{j_2}\|_q^2\right\} \\
            &\quad + O(\varepsilon),
        \end{split}
    \end{equation*}
    for $j_0 \in A_J$ and $j_1,j_2 \in B_J$, chosen based on the supremum, and all $n > N$. By the
    second inequality in Theorem \ref{thm 1} part 1, there exists a $C > 0$ such that
    $A_2^q + C < (A_2^\pm)^q$. Since $\|\phi^{j_1}\|_2^2 \leq 1$ by \eqref{profile ortho},
    \[
        A_2^2\|\mathcal{E}\phi^{j_1}\|_q^{4/d} \leq A_2^q \|\phi^{j_1}\|_2^{4/d} < (A_2^\pm)^q - C.
    \]
    Applying the same logic to $\mathcal{E}_-\psi^{j_2}$ and taking $\varepsilon$ sufficiently
    small, this proves that the maximum is achieved by the first term and hence
    \begin{equation}\label{pair is ext}
        A_2^\pm - O(\varepsilon) \leq \|\mathcal{E}S_n^{j_0}\phi^{j_0} + \mathcal{E}_-R_n^{j_0}\psi^{j_0}\|_q.
    \end{equation}
    Since
    \[
        \|\mathcal{E}S_n^{j_0}\phi^{j_0} + \mathcal{E}_-R_n^{j_0}\psi^{j_0}\|_q \leq A_2^\pm
        (\|\phi^{j_0}\|_2^2 + \|\psi^{j_0}\|_2^2)^{1/2},
    \]
    we have $\|\phi^{j_0}\|_2^2 + \|\psi^{j_0}\|_2^2 = 1$ by taking $\varepsilon \rightarrow 0$. By
    \eqref{profile ortho},
    \begin{equation}\label{first convergence}
        \|f_n - S_n^{j_0}\phi^{j_0}\|_2 \rightarrow 0 \quad \text{and} \quad \|g_n - R_n^{j_0}\psi^{j_0}\|_2 \rightarrow 0
    \end{equation}
	
    Let $\{(\lambda_n, t_n, x_n, \xi_n)\}$ and $\{(\kappa_n, s_n, y_n, \eta_n)\}$ be the parameters
    for $\{S_n^{j_0}\}\subset \mathbf{S}_+$ and $\{R_n^{j_0}\} \subset \mathbf{S}_-$ respectively
    and let $\{T_n^{j_0}\} \subset \mathbf{T}_+$ and $\{U_n^{j_0}\} \subset \mathbf{T}_-$ be the
    associated $L^q$ symmetries. Let $r_n = \frac{\kappa_n}{\lambda_n}$ and
    \begin{equation}\label{theta def}
        \theta_n := -r_n^2 |\xi_n|^2s_n - r_ny_n \cdot \xi_n - 2r_ns_n\eta_n\cdot\xi_n +
        s_n|\eta_n|^2 + y_n\cdot\eta_n.
    \end{equation}
    By calculation,
    \begin{multline*}
        (U_n^{j_0})^{-1}T_n^{j_0}F(t,x) = r_n^\frac{d+2}{q}e^{i(r_n^2 t, r_nx)\cdot(|\xi_n + r_n^{-1}\eta_n|^2, \xi_n + r_n^{-1}\eta_n) - 2i(t,x)\cdot(|\eta_n|^2, \eta_n) + i\theta_n} \\
        F(r_n^2t + t_n - r_n^2 s_n, r_nx + x_n - r_n y_n + 2r_n^2(\xi_n + r_n^{-1}\eta_n)(t-s_n)).
    \end{multline*}
    Since $j_0 \in A_J$, Definition \ref{sym decouple condition} implies that there exist
    $r_0 \in (0,\infty)$, $(t_0, x_0) \in \R^{d+1}$, $\xi_0 \in \R^d$, $\theta_0 \in [0,2\pi)$, and
    a subsequence in $n$ along which
    \begin{enumerate}
        \item $\lim r_n = r_0$,
        \item $\lim \xi_n + r_n^{-1}\eta_n = \xi_0$,
        \item
        $\lim (t_n - r_n^2s_n, x_n - r_ny_n - 2r_n^2 s_n(\xi_n + r_n^{-1}\eta_n)) = (t_0, x_0)$, and
        \item $\lim e^{i\theta_n} = e^{i\theta_0}$.
    \end{enumerate}
    Let $V^{j_0} \in \mathbf{T}_+$ be the symmetry associated with the parameters $r_n^{-1}$,
    $\xi_0$, and $(t_0, x_0)$, and let $W^{j_0} \in \mathbf{S}_+$ be the corresponding $L^2$
    isometry. Assume $F, G \in C^\infty_{cpct}$. Then by the dominated convergence theorem,
    \[
        \lim \|(U_n^{j_0})^{-1}T_n^{j_0}F(t,x) - e^{-2i(t,x)\cdot(|\eta_n|^2, \eta_n)}
        e^{i\theta_0}V^{j_0}F(t,x)\|_q = 0.
    \]
    By density, we can extend this to $F, G \in L^q$.
	
    Let $\Phi := \mathcal{E}\phi^{j_0}$ and $\Psi := \mathcal{E}_- \psi^{j_0}$. By the triangle
    inequality,
    \[
        \lim \|T_n^{j_0}\Phi + U_n^{j_0}\Psi\|_q - \|e^{-2i(t,x)\cdot(|\eta_n|^2, \eta_n)}
        e^{i\theta_0}V^{j_0}\Phi + \Psi\|_q = 0.
    \]
    If $|\eta_n| \rightarrow \infty$,
    \[
        \lim \|e^{-2i(t,x)\cdot(|\eta_n|^2, \eta_n)} e^{i\theta_0}V^{j_0}\Phi + \Psi\|_q^q =
        \frac{1}{2\pi}\int_0^{2\pi} \int \left|e^{i\omega}|V^{j_0}\Phi| + |\Psi|\right|^qdtdxd\omega
    \]
    by Lemma \ref{lambda lemma}. However, by \cite[Lemma 6.1]{FLS}
    \[
        \frac{1}{2\pi}\int_0^{2\pi} \int \left|e^{i\omega}|V^{j_0}\Phi| + |\Psi|\right|^qdtdxd\omega
        \leq \frac{2^{q/2}}{\sqrt{\pi}} \cdot
        \frac{\Gamma(\frac{q+1}{2})}{\Gamma(\frac{q+2}{2})}\left(\|\Phi\|_q^2 +
            \|\Psi\|_q^2\right)^{q/2}.
    \]
    This contradicts the hypothesis on $A_2^\pm$ and \eqref{pair is ext} by taking $\varepsilon$
    sufficiently small. Therefore $|\eta_n| \not\rightarrow \infty$ and there exists
    $\eta_0 \in \R^d$ such that $\eta_n \rightarrow \eta_0$ along a subsequence.
	
    Let $\{K_n\} \subset \mathbf{S}_+ \cap \mathbf{S}_-$ be the symmetries associated with the
    parameters $\{(\kappa_n, s_n, y_n + 2s_n\eta_n, 0)\}$. Let
    $\{L_n\} \subset \mathbf{T}_+ \cap \mathbf{T}_-$ be the associated $L^q$ (note that the
    definitions of the $L^q$ symmetries coincide when the frequency translation is zero so the $L_n$
    are well-defined). Let $r_n = \frac{\kappa_n}{\lambda_n}$,
    \[
        \theta_n := -r_n^2 s_n |\xi_n|^2 - r_ny_n\cdot\xi_n - 2r_ns_n\eta_n\cdot\xi_n,
    \]
    and
    \[
        \omega_n := -s_n|\eta_n|^2 - y_n\cdot\eta_n.
    \]
    We calculate
    \begin{multline*}
        L_n^{-1}T_n^{j_0}\Phi(t,x) + L_n^{-1}U_n^{j_0}\Psi(t,x) \\
        = r_n^\frac{d+2}{q} e^{i\theta_n} e^{i(r_n^2t, r_n x)\cdot(|\xi_n|^2, \xi_n)} \\
        \Phi(r_n^2t + t_n - r_n^2 s_n, r_nx + 2r_n^2t\xi_n + (2r_n^2s_n(r_n^{-1}\eta_n - \xi_n) + x_n - r_ny_n)) \\
        + e^{i\omega_n} e^{i(t,x)\cdot(-|\eta_n|^2, \eta_n)}\Psi(t,x - 2t\eta_n).
    \end{multline*}
    By Definition \ref{sym decouple condition} and the fact that $\lim \eta_n = \eta_0$, there exist
    $r_0 \in (0,\infty)$, $(t_0, x_0) \in \R^{d+1}$, $\xi_0 \in \R^d$, $\theta_0 \in [0, 2\pi)$,
    $\omega_0 \in [0, 2\pi)$, and a subsequence in $n$ along which
    \begin{enumerate}
        \item $\lim r_n = r_0$,
        \item $\lim \xi_n = \xi_0$,
        \item
        $\lim (t_n - r_n^2s_n, x_n - r_n y_n + 2r_n^2s_n(r_n^{-1}\eta_n - \xi_n)) = (t_0, x_0)$,
        \item $\lim e^{i\theta_n} = e^{i\theta_0}$, and
        \item $\lim e^{i\omega_n} = e^{i\omega_0}$.
    \end{enumerate}
    Let $(W^{j_0}, Y^{j_0}) \in \mathbf{S}_+ \times \mathbf{T}_+$ be the symmetry pair associated
    with the parameters $(r_0^{-1}, t_0,x_0, \xi_0)$ and let
    $(X^{j_0}, Z^{j_0}) \in \mathbf{S}_- \times \mathbf{T}_-$ be associated with the parameters
    $(1, 0,0, \eta_0)$. By approximating $\phi^{j_0}$ and $\psi^{j_0}$ in $C^\infty_{cpct}$ and
    applying dominated convergence, the convergence in parameters implies that
    \[
        K_n^{-1}S_n^{j_0}\phi^{j_0} \rightarrow e^{i\theta_0}W^{j_0}\phi^{j_0} \quad \text{and}
        \quad K_n^{-1}R_n^{j_0} \phi^{j_0} \rightarrow e^{i\omega_0}X^{j_0}\psi^{j_0}.
    \]
    Let $f = e^{i\theta_0}W^{j_0}\phi^{j_0}$ and $g = e^{i\omega_0}X^{j_0}\psi^{j_0}$. By
    \eqref{first convergence},
    \[
        K_n^{-1} f_n \rightarrow f \quad \text{and} \quad K_n^{-1} g_n \rightarrow g,
    \]
    in $L^2$. Finally, by the continuity of $\mathcal{E}_\pm$ from $\ell^2(L^2)$ to $L^q$,
    \[
        \|\mathcal{E}f + \mathcal{E}_- g\|_q = A_2^\pm.
    \]
\end{proof}

\begin{proof}[Proof of Theorem \ref{thm 2} part 2]
    If
    \[
        \left(\frac{1}{\sqrt{\pi}} \cdot
            \frac{\Gamma(\frac{q+1}{2})}{\Gamma(\frac{q+2}{2})}\right)^{1/q} 2^{1/p'} A_2 < A_2^\pm,
    \]
    then the previous part proves that extremizers exist. Therefore, assume
    \[
        \left(\frac{1}{\sqrt{\pi}} \cdot
            \frac{\Gamma(\frac{q+1}{2})}{\Gamma(\frac{q+2}{2})}\right)^{1/q} 2^{1/p'} A_2 = A_2^\pm.
    \]

    Let $f \in L^2$ be such that $\|f\|_2 = 1$ and $\|\mathcal{E}f\|_q = A_2$ (\cite{Stovall}). Let
    $g_\theta(\xi) = e^{i\theta}\overline{f}(-\xi)$. Then by the identity
    $\mathcal{E}_- h = \overline{\mathcal{E}\widetilde{h}}$ and \cite[Lemma 6.1]{FLS},
    \begin{align*}
      \frac{1}{2\pi} \int_0^{2\pi} \int |\mathcal{E}f + \mathcal{E}_- g_\theta|^q dtdxd\theta
      &= \frac{1}{2\pi} \int_0^{2\pi} \int |\mathcal{E}f + e^{-i\theta}\overline{\mathcal{E}f}|^q dtdxd\theta \\
      &= \frac{2^q}{\sqrt{\pi}} \cdot \frac{\Gamma(\frac{q+1}{2})}{\Gamma(\frac{q+2}{2})} \|\mathcal{E}f\|_q^q \\
      &= \frac{2^q}{\sqrt{\pi}} \cdot \frac{\Gamma(\frac{q+1}{2})}{\Gamma(\frac{q+2}{2})} A_2^q.
    \end{align*}
    Since $\int |\mathcal{E}f + \mathcal{E}_- g_\theta|^q$ is continuous in $\theta$ by dominated
    convergence, there exists a $\theta_0 \in [0, 2\pi)$ such that
    \[
        \|\mathcal{E}f + \mathcal{E}_-g_{\theta_0}\|_q = 2 \left(\frac{1}{\sqrt{\pi}} \cdot
            \frac{\Gamma(\frac{q+1}{2})}{\Gamma(\frac{q+2}{2})}\right)^{1/q} A_2.
    \]
    Therefore, by the assumption on $A_2^\pm$,
    \[
        \frac{\|\mathcal{E}f + \mathcal{E}_-g_{\theta_0}\|_q}{(\|f\|_2^2 +
          \|g_{\theta_0}\|_2^2)^{1/2}} = A_2^\pm.
    \]
\end{proof}

\section{Validated numerics}

\begin{proof}[Proof of Theorem \ref{thm 1} part 2]
    Let $f(\xi) = e^{-|\xi|^2}$. By \cite[Theorem 1.1]{Foschi 2007},
    $\|\mathcal{E}f\|_q = A_2\|f\|_2$. By the identity
    $\mathcal{E}_- f = \overline{\mathcal{E}\widetilde{f}}$ and \cite[Lemma 6.1]{FLS},
    \begin{align*}
      \frac{1}{2\pi} \int_0^{2\pi} \int \left|\Re e^{i\theta/2}\mathcal{E}f\right|^q dtdxd\theta &= \frac{1}{2\pi} \int_0^{2\pi} \int \left| e^{-i\theta/2}\mathcal{E}f + e^{i\theta/2}\mathcal{E}_- f\right|^q dtdxd\theta \\
                                                                                                 &= \frac{2^q}{\sqrt{\pi}} \cdot \frac{\Gamma(\frac{q+1}{2})}{\Gamma(\frac{q+2}{2})}A_2^q.
    \end{align*}
    Let $\mathcal{J}(\theta) = \int \left|\Re e^{i\theta} \mathcal{E}f\right|^q dtdx$. If we can
    show that $\mathcal{J}(\theta_1) \neq \mathcal{J}(\theta_2)$ for some $\theta_1 \neq \theta_2$,
    then by the continuity of $\mathcal{J}$ there exists a $\theta_0$ such that
    \begin{equation*}
        \mathcal{J}(\theta_0) > \frac{2^q}{\sqrt{\pi}} \cdot \frac{\Gamma(\frac{q+1}{2})}{\Gamma(\frac{q+2}{2})}A_2^q,
    \end{equation*}
    which will complete the proof.
	
    Expanding $\mathcal{E}f$ and integrating in polar coordinates, we see that
    \begin{align*}
      \mathcal{J}(\theta) &= \pi^{dq/2} \int (1+t^2)^{-dq/2} e^{\frac{-q|x|^2}{4(1+t^2)}} \left|\cos \left(\theta + \frac{t|x|^2}{4(1+t^2)}\right) \right|^q dtdx \\
                          &= c_d \int_{-\infty}^\infty \int_0^\infty (1+t^2)^\frac{d(1-q)}{2} r^{d-1} e^{-qr^2} \left| \cos \left( \theta + \frac{tr^2}{4}\right)\right|^q drdt
    \end{align*}
    where $c_d = \frac{2\pi^{d(1+q)/2}}{\Gamma(d/2)}$.
	
    We will consider $d=1$ as the case $d=2$ follows \textit{m.m.} Bounding $\cos$ by one and using
    Sage (\cite{sagemath}), we can bound the tail error
    \begin{equation*}
        c_d\int_{|t|>49} \int_4^\infty (1+t^2)^\frac{d(1-q)}{2} r^{d-1} e^{-qr^2} \left| \cos \left( \theta + \frac{tr^2}{4}\right)\right|^q drdt < 10^{-19}.
    \end{equation*}
    The techniques of interval arithmetic (e.g.\ \cite[Chapter 3]{Moore}, \texttt{RIF} Sage
    datatype) allow us to track potential computation errors in the upper and lower Riemann sums
    with steps of $0.1$. Sage provides
    \begin{equation*}
        c_d\int_{|t| < 50} \int_0^5 (1+t^2)^\frac{d(1-q)}{2} r^{d-1} e^{-qr^2} \left| \cos \left(\frac{tr^2}{4}\right)\right|^q drdt \in (23, 37)
    \end{equation*}
    and
    \begin{equation*}
        c_d\int_{|t| < 50} \int_0^5 (1+t^2)^\frac{d(1-q)}{2} r^{d-1} e^{-qr^2} \left| \cos \left( \frac{\pi}{2} + \frac{tr^2}{4}\right)\right|^q drdt \in (0, 0.1).
    \end{equation*}
    These intervals are further apart than twice the tail error, so
    $\mathcal{J}(0) \neq \mathcal{J}(\pi/2)$.
\end{proof}

\end{document}